\documentclass[11pt,reqno]{amsart}

\usepackage{caption,amsmath,amssymb,amsfonts,amsthm,latexsym,graphicx,multirow,color,hyperref,enumerate}
\usepackage[all]{xy}

\numberwithin{table}{section}

\newcommand{\A}{\mathrm{A}}    \newcommand{\Aut}{\mathrm{Aut}}
 \newcommand{\bbF}{\mathbb{F}} 
     \newcommand{\calO}{\mathcal{O}}

\newcommand{\G}{\mathrm{G}}    \newcommand{\GaO}{\mathrm{\Gamma O}} \newcommand{\GaU}{\mathrm{\Gamma U}}     \newcommand{\GU}{\mathrm{GU}}
  
\newcommand{\J}{\mathrm{J}}

\newcommand{\M}{\mathrm{M}} 
\newcommand{\N}{\mathrm{N}}
\newcommand{\Nor}{\mathbf{N}}

\newcommand{\Pa}{\mathrm{P}}        \newcommand{\POm}{\mathrm{P\Omega}}

  \newcommand{\SL}{\mathrm{SL}}  \newcommand{\Soc}{\mathrm{Soc}} \newcommand{\Sp}{\mathrm{Sp}}  \newcommand{\SU}{\mathrm{SU}}  \newcommand{\Sy}{\mathrm{S}}  
\newcommand{\Tr}{\mathrm{Tr}}

\def\Z{{\bf Z}}

\newtheorem{theorem}{Theorem}[section]
\newtheorem{lemma}[theorem]{Lemma}

\newtheorem{problem}[theorem]{Problem}

\theoremstyle{definition}

\newtheorem*{remark}{Remark}

\begin{document}

\title[Factorizations of almost simple orthogonal groups of minus type]{Factorizations of almost simple orthogonal groups of minus type}

\author[Li]{Cai Heng Li}
\address{(Li) SUSTech International Center for Mathematics and Department of Mathematics\\Southern University of Science and Technology\\Shenzhen 518055\\Guangdong\\P.~R.~China}
\email{lich@sustech.edu.cn}

\author[Wang]{Lei Wang}
\address{(Wang) School of Mathematics and Statistics\\Yunnan University\\Kunming 650091\\Yunnan\\P.~R.~China}
\email{wanglei@ynu.edu.cn}

\author[Xia]{Binzhou Xia}
\address{(Xia) School of Mathematics and Statistics\\The University of Melbourne\\Parkville 3010\\VIC\\Australia}
\email{binzhoux@unimelb.edu.au}

\begin{abstract}
This is the fourth one in a series of papers classifying the factorizations of almost simple groups with nonsolvable factors.
In this paper we deal with orthogonal groups of minus type.


\end{abstract}

\maketitle

\section{Introduction}

An expression $G=HK$ of a group $G$ as the product of subgroups $H$ and $K$ is called a \emph{factorization} of $G$, where $H$ and $K$ are called \emph{factors}. A group $G$ is said to be \emph{almost simple} if $S\leqslant G\leqslant\Aut(S)$ for some nonabelian simple group $S$, where $S=\Soc(G)$ is the \emph{socle} of $G$. In this paper, by a factorization of an almost simple group we mean that none its factors contains the socle. The main aim of this paper is to solve the long-standing open problem:

\begin{problem}\label{PrbXia1}
Classify factorizations of finite almost simple groups.
\end{problem}

Determining all factorizations of almost simple groups is a fundamental problem in the theory of simple groups, which was proposed by Wielandt~\cite[6(e)]{Wielandt1979} in 1979 at The Santa Cruz Conference on Finite Groups. It also has numerous applications to other branches of mathematics such as combinatorics and number theory, and has attracted considerable attention in the literature.
%Results on group factorizations also have numerous applications to other branches of mathematics such as combinatorics~\cite{CJT2007,Li2006,LP2008,LX}, number theory~\cite{FGS1993,GS1995} and Hopf algebra~\cite{EGGS2000,Kac1968,Takeuchi1981}.
In what follows, all groups are assumed to be finite if there is no special instruction.

The factorizations of almost simple groups of exceptional Lie type were classified by Hering, Liebeck and Saxl~\cite{HLS1987}\footnote{In part~(b) of Theorem~2 in~\cite{HLS1987}, $A_0$ can also be $\G_2(2)$, $\SU_3(3)\times2$, $\SL_3(4).2$ or $\SL_3(4).2^2$ besides $\G_2(2)\times2$.} in 1987.
For the other families of almost simple groups, a landmark result was achieved by Liebeck, Praeger and Saxl~\cite{LPS1990} thirty years ago, which classifies the maximal factorizations of almost simple groups. (A factorization is said to be \emph{maximal} if both the factors are maximal subgroups.)
Then factorizations of alternating and symmetric groups are classified in~\cite{LPS1990}, and factorizations of sporadic almost simple groups are classified in~\cite{Giudici2006}.
This reduces Problem~\ref{PrbXia1} to the problem on classical groups of Lie type.
Recently, factorizations of almost simple groups with a factor having at least two nonsolvable composition factors are classified in~\cite{LX2019}\footnote{In Table~1 of~\cite{LX2019}, the triple $(L,H\cap L,K\cap L)=(\Sp_{6}(4),(\Sp_2(4)\times\Sp_{2}(16)).2,\G_2(4))$ is missing, and for the first two rows $R.2$ should be $R.P$ with $P\leqslant2$.}, and those with a factor being solvable are described in~\cite{LX} and~\cite{BL2021}.

As usual, for a finite group $G$, we denote by $G^{(\infty)}$ the smallest normal subgroup of $X$ such that $G/G^{(\infty)}$ is solvable.
For factorizations $G=HK$ with nonsolvable factors $H$ and $K$ such that $L=\Soc(G)$ is a classical group of Lie type, the triple $(L,H^{(\infty)},K^{(\infty)})$ is described in~\cite{LWX}. Based on this work, in the present paper we characterize the triples $(G,H,K)$ such that $G=HK$ with $H$ and $K$ nonsolvable, and $G$ is an orthogonal group of minus type.

For groups $H,K,X,Y$, we say that $(H,K)$ contains $(X,Y)$ if $H\geqslant X$ and $K\geqslant Y$, and that $(H,K)$ \emph{tightly contains} $(X,Y)$ if in addition $H^{(\infty)}=X^{(\infty)}$ and $K^{(\infty)}=Y^{(\infty)}$. 
Our main result is the following Theorem~\ref{ThmOmegaMinus}.
Note that it is elementary to determine the factorizations of $G/L$ as this group has relatively simple structure (and in particular is solvable).

\begin{theorem}\label{ThmOmegaMinus}
Let $G$ be an almost simple group with socle $L=\POm_{2m}^-(q)$, where $m\geqslant4$, and let $H$ and $K$ be nonsolvable subgroups of $G$ not containing $L$. Then $G=HK$ if and only if (with $H$ and $K$ possibly interchanged) $G/L=(HL/L)(KL/L)$ and $(H,K)$ tightly contains $(X^\alpha,Y^\alpha)$ for some $(X,Y)$ in Table~$\ref{TabOmegaMinus}$ and $\alpha\in\Aut(L)$.
\end{theorem}

\begin{remark}
Here are some remarks on Tables~\ref{TabOmegaMinus}:
\begin{enumerate}[{\rm(I)}]
\item The column $Z$ gives the smallest almost simple group with socle $L$ that contains $X$ and $Y$. In other words, $Z=\langle L,X,Y\rangle$.
It turns out that $Z=XY$ for all pairs $(X,Y)$.
\item The groups $X$, $Y$ and $Z$ are described in the corresponding lemmas whose labels are displayed in the last column. 
%In these lemmas one can also read off (or easy to see from there) their intersection $X\cap Y$.
\item The description of groups $X$ and $Y$ are up to conjugations in $Z$ (see Lemma~\ref{LemXia04}(b) and Lemma~\ref{LemXia03}).
\end{enumerate}
\end{remark}

\begin{table}[htbp]
\captionsetup{justification=centering}
\caption{$(X,Y)$ for orthogonal groups of minus type}\label{TabOmegaMinus}
\begin{tabular}{|l|l|l|l|l|l|}
\hline
Row & $Z$ & $X$ & $Y$ & Remark & Lemma\\
\hline
1 & $\POm_{2m}^-(q)$ & $\Omega_{2m-1}(q)$ & $\SU_m(q)$ & $m$ odd & \ref{LemOmegaMinus01}\\
2 & $\POm_{2m}^-(q)$ & $q^{2m-2}{:}\Omega_{2m-2}^-(q)$ & $\SU_m(q)$ & $m$ odd & \ref{LemOmegaMinus02}\\
\hline
3 & $\Omega_{2m}^-(2)$ & $\Omega_{2m-2}^-(2).2$ & $\SU_m(2)$ & $m$ odd & \ref{LemOmegaMinus03}\\
4 & $\mathrm{O}_{2m}^-(2)$ & $\Omega_{2m-2}^-(2).2$ & $\SU_m(2).2$ & $m$ odd & \ref{LemOmegaMinus04}\\
5 & $\GaO_{2m}^-(4)$ & $\Omega_{2m-2}^-(4).4$ & $\SU_m(4).4$ & $m$ odd & \ref{LemOmegaMinus04}\\
\hline
6 & $\mathrm{O}_{2m}^-(2)$ & $\GaO_m^-(4)$ & $\Omega_{2m-1}(2).2$ & $m$ even & \ref{LemOmegaMinus05}\\
7 & $\GaO_{2m}^-(4)$ & $\GaO_m^-(16)$ & $\Omega_{2m-1}(4).4$ & $m$ even & \ref{LemOmegaMinus05}\\
\hline
8 & $\mathrm{O}_{2m}^-(2)$ & $\SU_{m/2}(4).4$ & $\Omega_{2m-1}(2).2$ & $m/2$ odd & \ref{LemOmegaMinus06}\\
9 & $\GaO_{2m}^-(4)$ & $\SU_{m/2}(16).8$ & $\Omega_{2m-1}(4).4$ & $m/2$ odd & \ref{LemOmegaMinus06}\\
\hline
10 & $\Omega_{10}^-(2)$ & $\A_{12}$, $\M_{12}$ & $2^8{:}\Omega_8^-(2)$ & & \ref{LemOmegaMinus07}\\
11 & $\Omega_{18}^-(2)$ & $3.\J_3$ & $2^{16}{:}\Omega_{16}^-(2)$ & & \ref{LemOmegaMinus08}\\
\hline
\end{tabular}
\vspace{3mm}
\end{table}

\section{Preliminaries}

In this section we collect some elementary facts regarding group factorizations.
%For a group $G$ and a subgroup $K$ of $G$, let $[G:K]$ denote the set of right cosets of $K$ in $G$.

\begin{lemma}\label{LemXia01}
Let $G$ be a group, let $H$ and $K$ be subgroups of $G$, and let $N$ be a normal subgroup of $G$. Then $G=HK$ if and only if $HK\supseteq N$ and $G/N=(HN/N)(KN/N)$.
\end{lemma}

\begin{proof}
If $G=HK$, then $HK\supseteq N$, and taking the quotient modulo $N$ we obtain
\[
G/N=(HN/N)(KN/N).
\]
Conversely, suppose that $HK\supseteq N$ and $G/N=(HN/N)(KN/N)$. Then 
\[
G=(HN)(KN)=HNK 
\]
as $N$ is normal in $G$. Since $N\subseteq HK$, it follows that $G=HNK\subseteq H(HK)K=HK$, which implies $G=HK$.
\end{proof}

Let $L$ be a nonabelian simple group. We say that $(H,K)$ is a \emph{factor pair} of $L$ if $H$ and $K$ are subgroups of $\Aut(L)$ such that $HK\supseteq L$. For an almost simple group $G$ with socle $L$ and subgroups $H$ and $K$ of $G$, Lemma~\ref{LemXia01} shows that $G=HK$ if and only if $G/L=(HL/L)(KL/L)$ and $(H,K)$ is a factor pair.
As the group $G/L$ has a simple structure (and in particular is solvable), it is elementary to determine the factorizations of $G/L$.
Thus to know all the factorizations of $G$ is to know all the factor pairs of $L$.
Note that, if $(H,K)$ is a factor pair of $L$, then any pair of subgroups of $\Aut(L)$ containing $(H,K)$ is also a factor pair of $L$.
Hence we have the following:

\begin{lemma}\label{LemXia02}
Let $G$ be an almost simple group with socle $L$, and let $H$ and $K$ be subgroups of $G$ such that $(H,K)$ contains some factor pair of $L$. Then $G=HK$ if and only if $G/L=(HL/L)(KL/L)$.
\end{lemma}

In light of Lemma~\ref{LemXia02}, the key to determine the factorizations of $G$ with nonsolvable factors is to determine the minimal ones (with respect to the containment) among factor pairs of $L$ with nonsolvable subgroups.
%The next lemma shows that conjugation of a factor pair of $L$ by elements of $L$ is still a factor pair.

\begin{lemma}\label{LemXia03}
Let $L$ be a nonabelian simple group, and let $(H,K)$ be a factor pair of $L$.
Then $(H^\alpha,K^\alpha)$ and $(H^x,K^y)$ are factor pairs of $L$ for all $\alpha\in\Aut(L)$ and $x,y\in L$.
\end{lemma}

\begin{proof}
It is evident that $H^\alpha K^\alpha=(HK)^\alpha\supseteq L^\alpha=L$. Hence $(H^\alpha,K^\alpha)$ is a factor pair.
Since $xy^{-1}\in L\subseteq HK$, there exist $h\in H$ and $k\in K$ such that $xy^{-1}=hk$. Therefore, 
\[
H^xK^y=x^{-1}Hxy^{-1}Ky=x^{-1}HhkKy=x^{-1}HKy\supseteq x^{-1}Ly=L,
\]
which means that $(H^x,K^y)$ is a factor pair.
\end{proof}

The next lemma is~\cite[Lemma~2(i)]{LPS1996}.

\begin{lemma}\label{LemXia05}
Let $G$ be an almost simple group with socle $L$, and let $H$ and $K$ be subgroups of $G$ not containing $L$. If $G=HK$, then $HL\cap KL=(H\cap KL)(K\cap HL)$. %is a factorization of the almost simple group $HL\cap KL$ with socle $L$.
\end{lemma}

The following lemma implies that we may consider specific representatives of a conjugacy class of subgroups when studying factorizations of a group.

\begin{lemma}\label{LemXia04}
Let $G=HK$ be a factorization. Then for all $x,y\in G$ we have $G=H^xK^y$ with $H^x\cap K^y\cong H\cap K$.
\end{lemma}
  
\begin{proof}
As $xy^{-1}\in G=HK$, there exists $h\in H$ and $k\in K$ such that $xy^{-1}=hk$. Thus
\[
H^xK^y=x^{-1}Hxy^{-1}Ky=x^{-1}HhkKy=x^{-1}HKy=x^{-1}Gy=G,
\]
and
\[
H^x\cap K^y=(H^{xy^{-1}}\cap K)^y\cong H^{xy^{-1}}\cap K=H^{hk}\cap K=H^k\cap K=(H\cap K)^k\cong H\cap K.\qedhere
\]
\end{proof}

\section{Notation}

Throughout this paper, let $q=p^f$ be a power of a prime $p$, let $m\geqslant4$ be an integer, let $\,\overline{\phantom{\varphi}}\,$ be the homomorphism from $\GaO_{2m}^-(q)$ to $\mathrm{P\Gamma O}_{2m}^-(q)$ modulo scalars, let $V$ be a vector space of dimension $2m$ over $\bbF_q$ equipped with a nondegenerate quadratic form $Q$, whose associated bilinear form is $\beta$, let $V_\sharp$ be a vector space of dimension $m$ over $\bbF_{q^2}$ with the same underlying set as $V$, and let $\Tr$ be the trace of the field extension $\bbF_{q^2}/\bbF_q$. 
For $w\in V$ with $Q(w)\neq0$, let $r_w$ be the reflection in $w$ defined by
\[
r_w\colon V\to V,\quad v\mapsto v-\frac{\beta(v,w)}{Q(w)}w.
\]

We fix the notation for some field extension subgroups of $\mathrm{O}(V)$ as follows, according to the parity of $m$. 

First assume that $m=2\ell+1$ is odd. In this case we may equip $V_\sharp$ with a nondegenerate unitary form $\beta_\sharp$ such that $Q(v)=\beta_\sharp(v,v)$ for all $v\in V$, and thus $\GU(V_\sharp)<\mathrm{O}(V)$. Take a standard $\bbF_{q^2}$-basis $E_1,F_1,\dots,E_\ell,F_\ell,D$ for $V_\sharp$, so that
\[
\beta_\sharp(E_i,E_j)=\beta_\sharp(F_i,F_j)=\beta_\sharp(E_i,D)=\beta_\sharp(F_i,D)=0,\quad\beta_\sharp(E_i,F_j)=\delta_{i,j},\quad\beta_\sharp(D,D)=1
\]
for all $i,j\in\{1,\dots,\ell\}$. Let $\psi\in\GaU(V_\sharp)$ such that 
\[
\psi\colon a_1E_1+b_1F_1+\dots+a_\ell E_\ell+b_\ell F_\ell+cD\mapsto a_1^pE_1+b_1^pF_1+\dots+a_\ell^pE_\ell+b_\ell^pF_\ell+c^pD
\]
for $a_1,b_1\dots,a_\ell,b_\ell,c\in\bbF_{q^2}$, and let $\lambda\in\bbF_{q^2}$ such that $\lambda+\lambda^q=1$. Then we have
\begin{align}
&Q(\lambda E_1)=\beta_\sharp(\lambda E_1,\lambda E_1)=0,\label{EqnOmegaMinus01}\\
&Q(F_1)=\beta_\sharp(F_1,F_1)=0,\label{EqnOmegaMinus02}\\
&Q(\lambda E_1+F_1)=\beta_\sharp(\lambda E_1+F_1,\lambda E_1+F_1)=\lambda+\lambda^q=1,\nonumber
\end{align}
and hence
\begin{equation}\label{EqnOmegaMinus03}
\beta(\lambda E_1,F_1)=Q(\lambda E_1+F_1)-Q(\lambda E_1)-Q(F_1)=1. 
\end{equation}
From~\eqref{EqnOmegaMinus01}--\eqref{EqnOmegaMinus03} we see that $(\lambda E_1,F_1)$ is a hyperbolic pair with respect to $Q$. 
Thus there exists a standard basis $e_1,f_1,\dots,e_{m-1},f_{m-1},d,d'$ for $V$ as in~\cite[2.2.3]{LPS1990} such that 
\[
e_1=\lambda E_1\ \text{ and }\ f_1=F_1. 
\]
Then $Q(e_1+f_1)=\beta_\sharp(e_1+f_1,e_1+f_1)=1$, which implies that $e_1+f_1$ is a nonsingular vector in both $V$ (with respect to $Q$) and $V_\sharp$ (with respect to $\beta_\sharp$). Let $\phi\in\GaO(V)$ be as defined in~\cite[\S2.8]{KL1990} with respect to the basis $e_1,f_1,\dots,e_{m-1},f_{m-1},d,d'$. 
Then $\phi$ fixes $e_1,f_1,\dots,e_{m-1},f_{m-1},d$, and commutes with $r_{e_1+f_1}$.

Next assume that $m=2\ell$ is even. In this case we may equip $V_\sharp$ with a nondegenerate quadratic form $Q_\sharp$ of minus type such that $Q(v)=\Tr(Q_\sharp(v))$ for all $v\in V$, and thus $\mathrm{O}(V_\sharp)<\mathrm{O}(V)$. Let $\beta_\sharp$ be the associated symmetric bilinear linear form of $Q_\sharp$, and take a standard $\bbF_{q^2}$-basis $E_1,F_1,\dots,E_{\ell-1},F_{\ell-1},D,D'$ for $V_\sharp$, so that
\begin{align*}
&Q_\sharp(E_i)=Q_\sharp(F_i)=0,\quad Q_\sharp(D)=1=\beta_\sharp(D,D'),\quad Q_\sharp(D')=\mu,\quad\beta_\sharp(E_i,F_j)=\delta_{i,j},\\
&\beta_\sharp(E_i,E_j)=\beta_\sharp(F_i,F_j)=\beta_\sharp(E_i,D)=\beta_\sharp(F_i,D)=\beta_\sharp(E_i,D')=\beta_\sharp(F_i,D')=0
\end{align*}
for all $i,j\in\{1,\dots,\ell-1\}$, where $\mu\in\bbF_{q^2}$ such that $x^2+x+\mu$ is irreducible over $\bbF_{q^2}$.

\section{Infinite families of $(X,Y)$ in Table~\ref{TabOmegaMinus}}\label{SecOmegaMinus01}

In the first lemma we construct the factor pair $(X,Y)$ in Row~1 of Table~\ref{TabOmegaMinus}.

\begin{lemma}\label{LemOmegaMinus01}
Let $G=\Omega(V)=\Omega_{2m}^-(q)$ with odd $m$, let $H=G_{e_1+f_1}$, let $K=\SU(V_\sharp)$, let $Z=\overline{G}$, let $X=\overline{H}$, and let $Y=\overline{K}$. Then $H\cap K=\SU_{m-1}(q)$, and $Z=XY$ with $Z=\POm_{2m}^-(q)$, $X\cong H=\Omega_{2m-1}(q)$ and $Y\cong K=\SU_m(q)$.
\end{lemma}

\begin{proof}
It is clear that $Z=\POm_{2m}^-(q)$, and $Y\cong K=\SU_m(q)$ as $m$ is odd. Since $e_1+f_1$ is a nonsingular vector in both $V$ (with respect to $Q$) and $V_\sharp$ (with respect to $\beta_\sharp$), we have $X\cong H=G_{e_1+f_1}=\Omega_{2m-1}(q)$ and
\[
H\cap K=G_{e_1+f_1}\cap\SU(V_\sharp)=\SU(V_\sharp)_{e_1+f_1}=\SU_{m-1}(q).
\]
It follows that
\[
\frac{|K|}{|H\cap K|}=\frac{|\SU_m(q)|}{|\SU_{m-1}(q)|}=q^{m-1}(q^m+1)=\frac{|\Omega_{2m}^-(q)|}{|\Omega_{2m-1}(q)|}=\frac{|G|}{|H|},
\]
and so $G=HK$. Hence $Z=\overline{G}=\overline{H}\,\overline{K}=XY$.
\end{proof}

The second lemma is on the factor pair $(X,Y)$ in Row~2 of Table~\ref{TabOmegaMinus}.

\begin{lemma}\label{LemOmegaMinus02}
Let $G=\Omega(V)=\Omega_{2m}^-(q)$ with odd $m$, let $H=G_{E_1}$, let $K=\SU(V_\sharp)$, let $Z=\overline{G}$, let $X=\overline{H}$, and let $Y=\overline{K}$. Then $H\cap K=(q.q^{2m-4}){:}\SU_{m-2}(q)$, and $Z=XY$ with $Z=\POm_{2m}^-(q)$, $X\cong H=q^{2m-2}{:}\Omega_{2m-2}^-(q)$ and $Y\cong K=\SU_m(q)$.
\end{lemma}

\begin{proof}
It is clear that $Z=\POm_{2m}^-(q)$. Since $m$ is odd, we have $Y\cong K=\SU_m(q)$. Since $Q(E_1)=\beta_\sharp(E_1,E_1)=0$, we see that $E_1$ is a singular vector in both $V$ (with respect to $Q$) and $V_\sharp$ (with respect to $\beta_\sharp$). Thus $X\cong H=G_{E_1}=q^{2m-2}{:}\Omega_{2m-2}^-(q)$ and
\[
H\cap K=G_{E_1}\cap\SU(V_\sharp)=\SU(V_\sharp)_{E_1}=(q.q^{2m-4}){:}\SU_{m-2}(q).
\]
It follows that
\[
\frac{|K|}{|H\cap K|}=\frac{|\SU_m(q)|}{|(q.q^{2m-4}){:}\SU_{m-2}(q)|}=(q^m+1)(q^{m-1}-1)=\frac{|\Omega_{2m}^-(q)|}{|q^{2m-2}{:}\Omega_{2m-2}^-(q)|}=\frac{|G|}{|H|},
\]
and so $G=HK$. Therefore, $Z=\overline{G}=\overline{H}\,\overline{K}=XY$.
\end{proof}

In the following two lemmas we construct the factor pairs $(X,Y)$ in Rows~3--5 of Table~\ref{TabOmegaMinus}.
Recall the notation of the field extension subgroups of $\mathrm{O}(V)$ for odd $m$ given at the start of this section.

\begin{lemma}\label{LemOmegaMinus03}
Let $Z=\Omega(V)=\Omega_{2m}^-(2)$ with $q=2$ and $m$ odd, let $X=Z_{\{e_1,f_1\}}$, and let $Y=\SU(V_\sharp)$. Then $Z=XY$ with $X=\Omega_{2m-2}^-(2).2$, $Y=\SU_m(2)$ and $X\cap Y=\SU_{m-2}(2)$.
\end{lemma}

\begin{proof}
Clearly, $Y=\SU_m(2)$. Since $(e_1,f_1)$ is a hyperbolic pair with respect to $Q$, we have $Z_{e_1,f_1}=\Omega_{2m-2}^-(q)$ and $Z_{\{e_1,f_1\}}=Z_{e_1,f_1}{:}\langle s\rangle$ for some $s$ swapping $e_1$ and $f_1$. In particular, $X=\Omega_{2m-2}^-(2).2$.

Suppose that $Y_{\{e_1,f_1\}}\neq Y_{e_1,f_1}$, that is, there exists $t\in Y$ swapping $e_1=\lambda E_1$ and $f_1=F_1$. Then as $t\in Y=\SU(V_\sharp)$, we obtain
\[
\lambda=\beta_\sharp(\lambda E_1,F_1)=\beta_\sharp((\lambda E_1)^t,F_1^t)=\beta_\sharp(F_1,\lambda E_1)=\lambda^q,
\]
contradicting the condition $\lambda+\lambda^q=1$. 

Thus we conclude that $Y_{\{e_1,f_1\}}=Y_{e_1,f_1}$. Consequently, 
\[
X\cap Y=Z_{\{e_1,f_1\}}\cap Y=Y_{\{e_1,f_1\}}=Y_{e_1,f_1}=\SU(V_\sharp)_{e_1,f_1}=\SU(V_\sharp)_{E_1,F_1}=\SU_{m-2}(2).
\]
As $q=2$, we then derive that
\[
\frac{|Y|}{|X\cap Y|}=\frac{|\SU_m(2)|}{|\SU_{m-2}(2)|}=2^{2m-3}(2^m+1)(2^{m-1}-1)=\frac{|\Omega_{2m}^-(2)|}{|\Omega^-_{2m-2}(2).2|}=\frac{|Z|}{|X|},
\]
which yields $Z=XY$.
\end{proof}

\begin{lemma}\label{LemOmegaMinus04}
Let $Z=\Omega(V){:}\langle\rho\rangle$ with $q\in\{2,4\}$, $m$ odd and $\rho\in\{\phi,(r_{e_1+f_1})^{2/f}\phi\}$, let $X=\Omega(V)_{e_1,f_1}{:}\langle\rho\rangle$, and let $Y=\SU(V_\sharp){:}\langle\psi\rangle$. Then $Z=XY$ with $Z=\GaO_{2m}^-(q)$, $X=\Omega_{2m-2}^-(q).(2f)$, $Y=\SU_m(q).(2f)$ and $X\cap Y=\SU_{m-2}(q)$.
\end{lemma}

\begin{proof}
Clearly, $Z=\GaO_{2m}^-(q)$, $X=\Omega_{2m-2}^-(q).(2f)$, and $Y=\SU_m(q).(2f)$. 

Suppose that $X\cap Y\nleqslant\Omega(V)$. Then there exist $s\in\Omega(V)_{e_1,f_1}$ and $t\in\SU(V_\sharp)$ such that $\phi^fs=\psi^ft$. Since $\phi$ and $s$ both fix $e_1$ and $f_1$, it follows that $\phi^fs$ fixes $e_1=\lambda E_1$ and $f_1=F_1$. This together with $t\in\SU(V_\sharp)$ implies that
\begin{align*}
\lambda=\beta_\sharp(\lambda E_1,F_1)&=\beta_\sharp((\lambda E_1)^{\phi^fs},F_1^{\phi^fs})\\
&=\beta_\sharp((\lambda E_1)^{\psi^ft},F_1^{\psi^ft})=\beta_\sharp((\lambda^qE_1)^t,F_1^t)=\beta_\sharp(\lambda^qE_1,F_1)=\lambda^q,
\end{align*}
contradicting the condition $\lambda+\lambda^q=1$. 

Thus we conclude that $X\cap Y\leqslant\Omega(V)$. Accordingly, 
\[
X\cap Y=(X\cap\Omega(V))\cap Y=\Omega(V)_{e_1,f_1}\cap Y=\SU(V_\sharp)_{e_1,f_1}=\SU(V_\sharp)_{E_1,F_1}=\SU_{m-2}(q).
\]
Observe $q=2f$ as $q\in\{2,4\}$. We then derive that
\[
\frac{|Y|}{|X\cap Y|}=\frac{|\SU_m(q).(2f)|}{|\SU_{m-2}(q)|}=2fq^{2m-3}(q^m+1)(q^{m-1}-1)=\frac{|\GaO^-_{2m}(q)|}{|\Omega_{2m-2}^-(q).(2f)|}=\frac{|Z|}{|X|}.
\]
This yields $Z=XY$.
\end{proof}

In the next lemma we construct the factor pairs $(X,Y)$ in Rows~6--7 of Table~\ref{TabOmegaMinus}.
Recall the notation of the field extension subgroups of $\mathrm{O}(V)$ for even $m$ given at the start of this section.

\begin{lemma}\label{LemOmegaMinus05}
Let $Z=\GaO(V)=\GaO_{2m}^-(q)$ with $q\in\{2,4\}$ and $m$ even, let $X=\GaO(V_\sharp)$, and let $Y=Z_{D'}$. Then $Z=XY$ with $X=\GaO_m^-(q^2)$, $Y=\Omega_{2m-1}(q).(2f)$ and $X\cap Y=\Omega_{m-1}(q^2).2$.
\end{lemma}

\begin{proof}
It is clear that $X=\GaO_m^-(q^2)$ is maximal in $Z$. Since $x^2+x+\mu$ is irreducible, we have $\mu\notin\bbF_q$, that is, $\mu\neq\mu^q$. Then as $q$ is even, we obtain
\[
Q(D')=\Tr(Q_\sharp(D'))=\Tr(\mu)=\mu+\mu^q\neq0.
\] 
It follows that an element in $Z$ stabilizes $\langle D'\rangle_{\bbF_q}$ if and only if it fixes $D'$. Hence 
\[
Y=Z_{D'}=Z_{\langle D'\rangle_{\bbF_q}}=\N_1[Z]=\Omega_{2m-1}(q).(2f)
\]
is maximal in $Z$. Then~\cite[Theorem~A]{LPS1990} asserts that $Z=XY$, which implies
\[
|X\cap Y|=\frac{|X|}{|Z|/|Y|}=\frac{|\GaO_m^-(q^2)|}{|Z|/|\N_1[Z]|}=\frac{|\GaO_m^-(q^2)|}{|\Omega_{2m}^-(q)|/|\Omega_{2m-1}(q)|}=\frac{|\GaO_m^-(q^2)|}{q^{m-1}(q^m+1)}=2|\Omega_{m-1}(q)|
\]
as $q=2f$. Let $r$ be the reflection in $V_\sharp$ with respect to the vector $D'$. Then $r\in\mathrm{O}(V_\sharp)\setminus\Omega(V_\sharp)$, and $r$ fixes $D'$. This implies that 
\[
X\cap Y=\GaO(V_\sharp)\cap Z_{D'}\geqslant\Omega(V_\sharp){:}\langle r\rangle=\Omega_{m-1}(q^2).2,
\]
which together with $|X\cap Y|=2|\Omega_{m-1}(q)|$ leads to $X\cap Y=\Omega(V_\sharp){:}\langle r\rangle=\Omega_{m-1}(q^2).2$.
% Note that the set of vectors fixed by $r$ forms an $(m-1)$-subspace of $V_\sharp$ and hence a $2(m-1)$-subspace of $V$. 
% We conclude that $r\in\Omega(V)$, and so $X\cap Y=\Omega(V_\sharp){:}\langle r\rangle<\Omega(V)$.
\end{proof}

Now we give the factor pairs $(X,Y)$ in Rows~8--9 of Table~\ref{TabOmegaMinus}.

\begin{lemma}\label{LemOmegaMinus06}
Let $Z=\GaO(V)=\GaO_{2m}^-(q)$ with $q\in\{2,4\}$ and $m=2\ell$ for some odd $\ell$, let $X=\SU_\ell(q^2){:}(4f)<\GaO_m^-(q^2)<Z$, and let $Y=Z_{D'}$. Then $Z=XY$ with $Y=\Omega_{2m-1}(q).(2f)$ and $X\cap Y=\SU_{\ell-1}(q^2).2$.
\end{lemma}

\begin{proof}
Similarly as in the proof of Lemma~\ref{LemOmegaMinus05} we have $Y=\Omega_{2m-1}(q).(2f)$. Let $A=\GaO_m^-(q^2)$ be a maximal subgroup of $Z$ containing $X$. Then Lemma~\ref{LemOmegaMinus05} asserts that $Z=AY$ with $A\cap Y=\Omega_{m-1}(q^2).2$. Moreover, from Lemma~\ref{LemOmegaMinus01} we see that $X\cap Y=X\cap(A\cap Y)$ has a normal subgroup $\SU_{\ell-1}(q^2)$, and $A=X(A\cap Y)$. Thus $Z=X(A\cap Y)Y=XY$ and
\[
|X\cap Y|=|X\cap(A\cap Y)|=\frac{|X||A\cap Y|}{|A|}=\frac{|\SU_\ell(q^2){:}(4f)||\Omega_{m-1}(q^2).2|}{|\GaO_m^-(q^2)|}=2|\SU_{\ell-1}(q^2)|.
\]
Consequently, $X\cap Y=\SU_{\ell-1}(q^2).2$.
\end{proof}

\section{Sporadic cases of $(X,Y)$ in Table~\ref{TabOmegaMinus}}\label{SecOmegaMinus02}

In this section, we give the sporadic pairs $(X,Y)$ in Table~\ref{TabOmegaMinus}, namely, the pairs $(X,Y)$ in Rows~10--11 of Table~\ref{TabOmegaMinus}.

\begin{lemma}\label{LemOmegaMinus07}
Let $Z=\Omega_{10}^-(2)$, let $X$ be a subgroup of $Z$ isomorphic to $\A_{12}$ or $\M_{12}$ (there is a unique conjugacy class of such subgroups $X$ in each case), and let $Y=\Pa_1[Z]=2^8{:}\Omega_8^-(2)$. Then $Z=XY$ with 
\[
X\cap Y=
\begin{cases}
(\A_4\times\A_8).2&\textup{if }X=\A_{12}\\
2_+^{1+4}.\Sy_3&\textup{if }X=\M_{12}.
\end{cases}
\]
\end{lemma}

\begin{lemma}\label{LemOmegaMinus08}
Let $Z=\Omega_{18}^-(2)$, let $X=3.\J_3<\SU_9(2)<Z$, and let $Y=\Pa_1[Z]=2^{16}{:}\Omega_{16}^-(2)$. Then $Z=XY$ with $X\cap Y=2^{2+4}.(3\times\Sy_3)$.
\end{lemma}

\begin{proof}
Let $M=\SU_9(2)$ be a subgroup of $Z$ containing $X$. By Lemma~\ref{LemOmegaMinus02} we have $M\cap Y=2^{1+14}{:}\SU_7(2)$ with $M\cap Y\ngeqslant\Z(M)=3$.
Then it follows from Lemma~\ref{LemUnitary18} that
\[
X\cap Y=X\cap(M\cap Y)\cong(X/\Z(M))\cap((M\cap Y)\Z(M)/\Z(M))=2^{2+4}.(3\times\Sy_3).
\]
Therefore,
\[
\frac{|X|}{|X\cap Y|}=\frac{|3.\J_3|}{|2^{2+4}.(3\times\Sy_3)|}=(2^9+1)(2^8-1)=\frac{|\Omega_{18}^-(2)|}{|2^{16}{:}\Omega_{16}^-(2)|}=\frac{|Z|}{|Y|},
\]
and so $Z=XY$.
\end{proof}

\section{Proof of Theorem~\ref{ThmOmegaMinus}}

Let $G$ be an almost simple group with socle $L=\POm_{2m}^-(q)$, and let $H$ and $K$ be nonsolvable subgroups of $G$ not containing $L$.
In Subsections~\ref{SecOmegaMinus01} and~\ref{SecOmegaMinus02} it is shown that all pairs $(X,Y)$ in Table~\ref{TabOmegaMinus} are factor pairs of $L$.
Hence by Lemma~\ref{LemXia02} we only need to prove that, if $G=HK$, then $(H,K)$ tightly contains $(X^\alpha,Y^\alpha)$ for some $(X,Y)$ in Table~\ref{TabOmegaMinus} and $\alpha\in\Aut(L)$. Suppose that 
\[
G=HK. 
\]
Then by~\cite[Theorem~7.1]{LWX} the triple $(L,H^{(\infty)},K^{(\infty)})$ lies in Table~\ref{TabInftyOmegaMinus}.

\begin{table}[htbp]
\captionsetup{justification=centering}
\caption{$(L,H^{(\infty)},K^{(\infty)})$ for orthogonal groups of minus type}\label{TabInftyOmegaMinus}
\begin{tabular}{|l|l|l|l|l|l|}
\hline
Row & $L$ & $H^{(\infty)}$ & $K^{(\infty)}$ & Conditions\\
\hline
1 & $\POm_{2m}^-(q)$ & $\Omega_{2m-1}(q)$, $\Omega^-_{2m-2}(q)$, $q^{2m-2}{:}\Omega^-_{2m-2}(q)$ & $\SU_m(q)$ & $m$ odd\\
2 & $\Omega_{2m}^-(2)$ & $\SU_{m/2}(4)$ ($m/2$ odd), $\Omega_m^-(4)$, & $\Sp_{2m-2}(2)$ & \\
 & & $\SU_{m/4}(16)$ ($m/4$ odd), $\Omega_{m/2}^-(16)$ & & \\
3 & $\Omega_{2m}^-(4)$ & $\SU_{m/2}(16)$ ($m/2$ odd), $\Omega_m^-(16)$ & $\Sp_{2m-2}(4)$ & \\
\hline
4 & $\Omega_{10}^-(2)$ & $\A_{12}$, $\M_{12}$ & $2^8{:}\Omega_8^-(2)$ & \\
5 & $\Omega_{18}^-(2)$ & $3.\J_3$ & $2^{16}{:}\Omega_{16}^-(2)$ & \\
\hline
\end{tabular}
\vspace{3mm}
\end{table}

\begin{lemma}\label{LemOmegaMinus09}
Suppose that $H^{(\infty)}=\Omega_{2m-2}^-(q)$ and $K^{(\infty)}=\SU_m(q)$ with $m$ odd. Then $q\in\{2,4\}$, and $(H,K)$ tightly contains some pair $(X,Y)$ in Rows~\emph{3--5} of Table~$\ref{TabOmegaMinus}$.
\end{lemma}

\begin{proof}
Let $M=\Nor_G(H^{(\infty)})$. Then $H\leqslant M=\N_2^+[G]$, and so $G=MK$. Note from~\cite[Table~3.5.F]{KL1990} that $M$ is maximal in $G$ for $q\geqslant4$. Then by~\cite[Theorem~A]{LPS1990} we have $q\leqslant4$. Without loss of generality, we may assume $H^{(\infty)}=\Omega(V)_{e_1,f_1}$ and $K^{(\infty)}=\SU(V_\sharp)$. Then
\[
H\cap K\geqslant H^{(\infty)}\cap K^{(\infty)}=\SU(V_\sharp)_{e_1,f_1}=\SU(V_\sharp)_{E_1,F_1}=\SU_{m-2}(q).
\]
Let $\calO=G/L$, $\calO_1=H/H^{(\infty)}$ and $\calO_2=K/K^{(\infty)}$. It follows that
\[
|H^{(\infty)}|_p|\calO_1|_p|K^{(\infty)}|_p|\calO_2|_p=|H|_p|K|_p=|G|_p|H\cap K|_p\geqslant|L|_p|\calO|_p|\SU_{m-2}(q)|_p,
\]
and thus
\begin{equation}\label{EqnOmegaMinus04}
|\calO_1|_p|\calO_2|_p\geqslant\frac{|L|_p|\SU_{m-2}(q)|_p|\calO|_p}{|H^{(\infty)}|_p|K^{(\infty)}|_p}
=\frac{|\Omega_{2m}^-(q)|_p|\SU_{m-2}(q)|_p|\calO|_p}{|\Omega_{2m}^-(q)|_p|\SU_m(q)|_p}=q|\calO|_p.
\end{equation}
Note from~\cite[Proposition~4.3.18]{KL1990} that $|K\cap L|_p=|K^{(\infty)}|_p$. 
Hence $|\calO|_p\geqslant|KL/L|_p=|K/K\cap L|_p=|\calO_2|_p$, and so~\eqref{EqnOmegaMinus04} yields $|\calO_1|_p\geqslant q$.

First assume $q=2$. Then $|\calO_1|_2\geqslant2$, and so $H\geqslant X$ for some $X=\Omega_{2m-2}^-(2).2$. 
Applying this conclusion to the factorization $HL\cap KL=(H\cap KL)(K\cap HL)$ we also obtain $H\cap KL\geqslant\Omega_{2m-2}^-(2).2$.
If $X\leqslant L$, then $(H,K)$ tightly contains the pair $(X,Y)$ with $Y=K^{(\infty)}$, as in Row~3 of Table~\ref{TabOmegaMinus}.
If $X\nleqslant L$ and $K\leqslant L$, then $H\cap L=H\cap KL\geqslant\Omega_{2m-2}^-(2).2$ and so we are led to the previous case where $X\leqslant L$.
If $X\nleqslant L$ and $K\nleqslant L$, then $(H,K)$ tightly contains the pair $(X,Y)$ with $Y=\SU_m(2).2$, as in Row~4 of Table~\ref{TabOmegaMinus}.

Next assume $q=3$. Then $|H|_3\leqslant|M|_3=|\Omega_{2m-2}^-(3)|_3=|H^{(\infty)}|_3$, which implies that $|\calO_1|_3=1$, contradicting the conclusion $|\calO_1|_p\geqslant q$.

Finally assume $q=4$. By~\cite[Theorem~A]{LPS1990} we have $G=\Aut(L)=\Omega_{2m-2}^-(4).4$. This conclusion applied to the factorization $HL\cap KL=(H\cap KL)(K\cap HL)$ leads to $HL=KL=\Omega_{2m-2}^-(4).4$. It follows that $H=(H\cap L).4$ and $K=(K\cap L).4$. Hence $(H,K)$ tightly contains the pair $(X,Y)$ in Row~5 of Table~\ref{TabOmegaMinus}.
\end{proof}

By virtue of Lemma~\ref{LemOmegaMinus09}, if $(L,H^{(\infty)},K^{(\infty)})$ lies in Row~1 of Table~\ref{TabInftyOmegaMinus}, then $(H,K)$ tightly contains some pair $(X,Y)$ in Rows~1--5 of Table~\ref{TabOmegaMinus}.

If $(L,H^{(\infty)},K^{(\infty)})$ lies in Row~2 or~3 of Table~\ref{TabInftyOmegaMinus}, then the next two lemmas show that $(H,K)$ tightly contains $(X,Y)$ in Rows~6--9 of Table~\ref{TabOmegaMinus}.

\begin{lemma}
Suppose that $K^{(\infty)}=\Sp_{2m-2}(q)$ with $q\in\{2,4\}$ and either $H^{(\infty)}=\SU_{m/2}(q^2)$ with $m/2$ odd or $H^{(\infty)}=\Omega_m^-(q^2)$. Then $(H,K)$ tightly contains some pair $(X,Y)$ in Rows~\emph{6--9} of Table~$\ref{TabOmegaMinus}$.
\end{lemma}

\begin{proof}
Since $K^{(\infty)}=\Sp_{2m-2}(q)$, we have $K\leqslant B$ for some $B=\N_1[G]$. Since either $H^{(\infty)}=\SU_{m/2}(q^2)$ with $m/2$ odd or $H^{(\infty)}=\Omega_m^-(q^2)$, there exists a maximal subgroup $A$ of $G$ such that $A^{(\infty)}=\Omega_m^-(q^2)$. From $G=HK$ we deduce $G=AB$. Then~\cite[Theorem~A]{LPS1990} implies $G=\Aut(L)=\mathrm{O}_{2m}^-(q)$. Applying this conclusion to the factorization $HL\cap KL=(H\cap KL)(K\cap HL)$ we obtain $HL=KL=\mathrm{O}_{2m}^-(q)$. Hence $H/H\cap L\cong K/K\cap L=2f$, and so $(H,K)$ tightly contains some pair $(X,Y)$ in Rows~6--9 of Table~\ref{TabOmegaMinus}.
\end{proof}

\begin{lemma}
If $q=2$ and either $H^{(\infty)}=\SU_{m/4}(16)$ with $m/4$ odd or $H^{(\infty)}=\Omega_{m/2}^-(16)$, then $K^{(\infty)}\ncong\Sp_{2m-2}(2)$.
\end{lemma}

\begin{proof}
Suppose for a contradiction that $K^{(\infty)}\cong\Sp_{2m-2}(2)$. Then $K\leqslant B$ for some $B=\N_1[G]$. Since either $H^{(\infty)}=\SU_{m/4}(16)$ with $m/4$ odd or $H^{(\infty)}=\Omega_{m/2}^-(16)$, there exist subgroups $M$ and $A$ of $G$ such that $H\leqslant M<A$ such that $A$ is a maximal subgroup of $G$ with $A^{(\infty)}=\Omega_m^-(4)$ and $M$ is a maximal subgroup of $A$ with $M^{(\infty)}=\Omega_{m/2}^-(16)$. From $G=HK$ we deduce $A=M(A\cap B)$ and $G=AB$. Then~\cite[Theorem~A]{LPS1990} implies $G=\Aut(L)=\mathrm{O}_{2m}^-(2)$. It follows that $A=\GaO_m^-(4)$ and $B=\Omega_{2m-1}(2).2$. Hence $A\cap B=\Omega_{m-1}(4).2$ as Lemma~\ref{LemOmegaMinus05} asserts. Let $N=(A\cap B)\Soc(A)$. Then since $A/\Soc(A)=4$, it follows that $N<A$. From $A=M(A\cap B)$ and $A\cap B\leqslant N$ we deduce
\[
N=(M\cap N)(A\cap B).
\]
Moreover, $\Soc(N)=\Soc(A)=\Omega_m^-(4)$, $(M\cap N)^{(\infty)}=\Omega_{m/2}^-(16)$, and $(A\cap B)^{(\infty)}=\Omega_{m-1}(4)$. Thus~\cite[Theorem~A]{LPS1990} implies $N=\GaO_m^-(4)$, contradicting the conclusion that $N<A$.
\end{proof}

Finally, if $(L,H^{(\infty)},K^{(\infty)})$ lies in Row~4 or~5 of Table~\ref{TabInftyOmegaMinus}, then the pair $(H,K)$ tightly contains $(X,Y)=(H^{(\infty)},K^{(\infty)})$ in Row~10 or~11, respectively, of Table~\ref{TabOmegaMinus}. This completes the proof of Theorem~\ref{ThmOmegaMinus}.

\section*{Acknowledgments}
The first author acknowledges the support of NNSFC grants no.~11771200 and no.~11931005. The second author acknowledges the support of NNSFC grant no.~12061083.

\end{document}